\documentclass[11pt]{amsart}
\usepackage{latexsym}
\usepackage{amssymb}
\usepackage{amsmath}

\newtheorem{theorem}{Theorem}[section]
\newtheorem{lemma}[theorem]{Lemma}
\newtheorem{proposition}[theorem]{Proposition}
\newtheorem{corollary}[theorem]{Corollary}

\newtheorem{defn}[theorem]{Definition}
\newtheorem{prop}[theorem]{Proposition}
\newtheorem{thm}[theorem]{Theorem}

\theoremstyle{definition}
\newtheorem{remark}[theorem]{Remark}

\newtheorem{question}[theorem]{Question}

\newcommand\id{\mathop{\rm id}}

\newcommand\nph{\varphi}

\newcommand\vecc{\mathop{\rm vect}}

\newcommand\comm{\mathop{\rm c}}

\newcommand\qmin{\mathop{\rm qmin}}
\newcommand\qmax{\mathop{\rm qmax}}
\newcommand\qc{\mathop{\rm qc}}
\newcommand\q{\mathop{\rm q}}
\newcommand\qs{\mathop{\rm qs}}

\newcommand\coisubset{\subseteq_{\rm coi}}   %%%%% coi inclusion

\newcommand{\cl}[1]{\mathcal{#1}}
\newcommand{\bb}[1]{\mathbb{#1}}

\begin{document}

\title{Quantum chromatic numbers via operator systems}

\author[V.~I. Paulsen]{Vern I. Paulsen}
\address{Department of Mathematics, University of Houston,
Houston, Texas 77204-3476, U.S.A.}
\email{vern@math.uh.edu}

\author[I.~G. Todorov]{Ivan G. Todorov}
\address{Pure Mathematics Research Centre, Queen's University Belfast, Belfast BT7 1NN, United Kingdom}
\email{i.todorov@qub.ac.uk}

\thanks{This work supported in part by NSF (USA), EPSRC (United Kingdom) and the 
Isaac Newton Institute for Mathematical Sciences}

\date{24 November 2013}

\begin{abstract}
We define several new types of quantum chromatic numbers of a graph and  
characterise them in terms of operator system tensor products.
We establish inequalities 
between these chromatic numbers and other parameters of graphs 
studied in the literature and exhibit a link 
between them and non-signalling correlation boxes.
\end{abstract}

\maketitle

\section{Introduction}\label{s_introd}

One of the most important parameters 
of a graph is its chromatic number  (see, {\it e.g.}, \cite{gr}). 
To recall its definition, let $G = (V,E)$ be a (finite, undirected) 
graph; here $V$ is the set of vertices, while $E$ is the set of edges, of $G$. 
A \emph{$d$-colouring} (where $d$ is a positive integer) is a map 
$r : V\to \{1,\dots,d\}$ such that if $(v,w)\in E$ then $r(v)\neq r(w)$. 
The \emph{chromatic number} $\chi(G)$ of $G$ is the smallest positive integer $d$ 
for which $G$ admits a $d$-colouring.

In \cite{AHKS06} and \cite{cnmsw}, the authors considered a \lq\lq graph colouring game'', where 
two players, Alice and Bob, try to convince a referee that they have a colouring of a graph $G$;
the referee inputs a pair $(v,w)$ of vertices of $G$, and each of the players 
produces an output, according to a 
previously agreed \lq\lq quantum strategy'', that is, a probability distribution 
derived from an entangled state and collections of POVM's. 
To formalise this, recall that a POVM is a collection $(E_i)_{i=1}^k$ of positive operators 
acting on a Hilbert space $H$ with $\sum_{i=1}^k E_i = I$ (here we denote as usual by $I$ 
the identity operator). When $H = \bb{C}^p$ is finite dimensional, 
we identify the operators on $H$ with elements of the algebra $M_p$ of all 
$p$ by $p$ complex matrices. 
Given POVM's $(E_{v,i})_{i=1}^c \subseteq M_p$ and 
$(F_{w,j})_{j=1}^c \subseteq M_q$, where $v,w \in V$, and 
a unit vector $\xi\in \bb{C}^p\otimes\bb{C}^q$, one associates with each pair 
$(v,w)$ of vertices of $G$ the probability distribution 
$(\langle (E_{v,i} \otimes F_{w,j}) \xi, \xi \rangle)_{i,j=1}^c$; here, 
for an input $(v,w)$ from the referee, $\langle (E_{v,i} \otimes F_{w,j}) \xi, \xi \rangle$
is the probability for Alice producing an output $i$ and Bob -- an output $j$. 
Alice and Bob can thus convince the referee that they have a $c$-colouring if the 
following conditions are satisfied:
\begin{multline}\label{qcdefn}
\forall v, \forall i \ne j, \langle (E_{v,i} \otimes F_{v,j}) \xi, \xi \rangle =0,\\
\forall (v,w) \in E, \forall i, \langle (E_{v,i} \otimes F_{w,i}) \xi, \xi \rangle =0.
\end{multline}
If this happens, we say that the graph $G$ admits a \emph{quantum $c$-colouring}; 
the smallest positive integer $c$ for which $G$ admits a quantum $c$-colouring was
called in \cite{cnmsw} the \emph{quantum chromatic number} of $G$ and denoted by $\chi_{\q}(G)$. 

It is easy to see that $\chi_{\q}(G)\leq \chi(G)$; indeed, let $d = \chi(G)$ and
choose a $d$-colouring $r : V\to \{1,\dots,d\}$ of $G$. 
For $v\in V$, let $E_{v,i} = F_{v,i} = 1$ if $i = r(v)$ and $E_{v,i} = F_{v,i} = 0$ if $i \neq r(v)$.
The families $(E_{v,i})_{i = 1}^d$ and $(F_{v,i})_{i = 1}^d$ are POVM's 
on the Hilbert space $\bb{C}$ for each $v\in V$, 
and the conditions (\ref{qcdefn}) are clearly satisfied.

In the present paper, we introduce other natural versions of the quantum chromatic number.
Our definitions are motivated by two different sets of axioms of quantum mechanics 
as well as by a probabilistic view on the graph coloring game.  
In the {\it non-relativisitc} viewpoint, measurement operators for different labs are assumed to 
act on different Hilbert spaces and joint measurements are obtained 
by taking the tensor product of these Hilbert spaces. 
But, unlike in the above definition of the quantum chromatic number, 
there is no requirement that these Hilbert spaces be finite dimensional.  
We hence consider the difference between finite and infinite dimensional Hilbert spaces. 
In the {\it relativistic} view, on the other hand, 
the measurement operators act on a common Hilbert space 
but are assumed to commute; the joint measurement operators are in this case
obtained by considering corresponding products. 
We also introduce a probabilistic view on the graph colouring game, where, 
instead of requiring the existence of quantum $c$-colourings, 
we assume that for a fixed number of $c$ colours,  
colouring strategies exist that win with probability $1- \epsilon$ for every $\epsilon > 0.$

Tsirelson \cite{tsirelson1980}, \cite{tsirelson1993} attempted to reconcile the outcomes of these two versions for POVM's. But we now know, thanks to the work \cite{jnppsw}, that equality of all matricial  outcomes for the two different versions is equivalent to Connes' Embedding Problem.

It is thus interesting to examine the behaviour of different versions of quantum chromatic numbers in order to see if the delicate combinatorics of graphs can shed any new light on these issues.

%The viewpoints outlined above lead to several versions of the quantum 
%chromatic number. 
The  viewpoint highlighted in the present paper is that of 
operator system tensor products \cite{kavruk--paulsen--todorov--tomforde2011}; 
we show how the different tensor products introduced in 
\cite{kavruk--paulsen--todorov--tomforde2011} can be used to characterise 
the chromatic numbers we introduce, as well as the classical chromatic number. 
We furthermore describe the new parameters 
geometrically, in a fashion that is a natural extension of \cite{cnmsw} and fits with Tsirelson's work on non-signalling boxes \cite{tsirelson1980}.

Throughout the paper, we will use notions and results about operator systems and their 
tensor products which can be found in \cite{kavruk--paulsen--todorov--tomforde2011}, 
\cite{kavruk2011} and \cite{kptt2010}. We refer to these papers about 
basic properties of the tensor products that will be used subsequently, 
namely, the minimal, the commuting and the maximal tensor product, 
which are denoted by $\otimes_{\min}$, $\otimes_{\rm c}$ and $\otimes_{\max}$, 
respectively. 
We recall that the minimal (resp. maximal) operator system structure on 
the algebraic tensor product $\cl S\otimes\cl T$ of two operator systems $\cl S$ and $\cl T$
is the one with the largest (resp. smallest) matricial cones. Note that, 
given unital complete order embeddings of $\cl S$ and $\cl T$ into 
$\cl B(H)$ and $\cl B(K)$, respectively, the operator system 
$\cl S\otimes_{\min}\cl T$ is obtained after embedding 
$\cl S\otimes\cl T$ into $\cl B(H\otimes K)$.
Furthemore, the operator system $\cl S\otimes_{\rm c}\cl T$ is characterised by the universal 
property that for any pair of unital completely positive maps
$\phi : \cl S\to \cl B(H)$ and $\psi : \cl T\to \cl B(H)$, with 
commuting ranges, the map 
$\phi\cdot \psi : \cl S\otimes_{\rm c}\cl T\to \cl B(H)$ given by $\phi\cdot\psi(x\otimes y) = \phi(x)\psi(y)$,
is completely positive.

The notion of a coproduct of operator systems will play an important role
hereafter; we refer the reader to \cite{fritz} and \cite{kavruk2011}. 
Finally, we recall that if $\cl S$ and $\cl T$ are operator systems with 
$\cl S$ contained, as a linear space, in $\cl T$, the notation 
$\cl S\subseteq_{\rm coi}\cl T$ means that the 
inclusion of $\cl S$ into $\cl T$ is a complete order isomorphism onto its range.

\section{Quantum chromatic numbers and their characterisations}\label{s_qcn}

We start by introducing the various chromatic numbers we are going to investigate in this paper. 
Throughout the paper, we let $n = |V|$.

\begin{defn}\label{d_3}
Let $G = (V,E)$ be a graph. 

(i) The \emph{approximate quantum chromatic number} $\chi_{\rm qa}(G)$ 
of $G$ is the smallest $c \in \bb N$ such that for each $\epsilon >0$ there exist $p,q\in \bb{N}$ and POVM's 
$(E_{v,i})_{i=1}^c \subseteq M_p$ and 
$(F_{w,j})_{j=1}^c \subseteq M_q$, where $v,w \in V$,  and a vector $\xi\in \bb{C}^p\otimes\bb{C}^q$ satisfying the conditions
\begin{multline}\label{qadefn}
\forall v, \forall i \ne j, \langle (E_{v,i} \otimes F_{v,j}) \xi, \xi \rangle < \epsilon,\\
\forall (v,w) \in E, \forall i, \langle (E_{v,i} \otimes F_{w,i}) \xi, \xi \rangle < \epsilon.
\end{multline}

(ii) The \emph{spacial quantum chromatic number} $\chi_{\qs}(G)$ of $G$
is the smallest $c\in \bb{N}$ for which there exist
Hilbert spaces $H$ and $K$, a unit vector $\xi\in H\otimes K$ and POVM's
$(E_{v,i})_{i=1}^c \subseteq \cl B(H)$ and $(F_{w,j})_{j=1}^c \subseteq \cl B(K)$,
where $v,w \in V$, satisfying conditions (\ref{qcdefn});

(iii) The \emph{approximate spacial quantum chromatic number} 
$\chi_{\rm qas}(G)$ of $G$ is the smallest $c \in \bb N$ such that for every $\epsilon >0,$ there exist
Hilbert spaces $H$ and $K$, a unit vector $\xi\in H\otimes K$ and POVM's
$(E_{v,i})_{i=1}^c \subseteq \cl B(H)$ and $(F_{w,j})_{j=1}^c \subseteq \cl B(K)$,
where $v,w \in V$, satisfying (\ref{qadefn});
%\begin{multline}\label{qasdefn}
%\forall v, \forall i \ne j, \langle (E_{v,i} \otimes F_{v,j}) \xi, \xi \rangle < \epsilon,\\
%\forall (v,w) \in E, \forall i, \langle (E_{v,i} \otimes F_{w,i}) \xi, \xi \rangle < \epsilon.
%\end{multline}

(iv) The \emph{relativistic quantum chromatic number} $\chi_{\rm qr}(G)$ of $G$
is the smallest $c\in \bb{N}$ for which there exists a Hilbert space $H,$ a unit vector $\xi \in H$ and POVM's
$(E_{v,i})_{i=1}^c \subseteq \cl B(H)$ and $(F_{w,j})_{j=1}^c \subseteq \cl B(H)$,
where $v,w \in V$, such that $E_{v,i}F_{w,j} = F_{w,j}E_{v,i}$ for all $v,w\in V$, $i,j = 1,\dots,c$, and 
\begin{multline}\label{qcreldefn}
\forall v, \forall i \ne j, \langle E_{v,i}  F_{v,j} \xi, \xi \rangle =0,\\
\forall (v,w) \in E, \forall i, \langle E_{v,i}  F_{w,i} \xi, \xi \rangle =0.
\end{multline}

(v) The \emph{approximate relativistic quantum chromatic number} $\chi_{\rm qar}(G)$ of $G$
is the smallest $c\in \bb{N}$ such that for every $\epsilon >0,$  
there exists a Hilbert space $H,$ a unit vector $\xi \in H$ and POVM's
$(E_{v,i})_{i=1}^c \subseteq \cl B(H)$ and $(F_{w,j})_{j=1}^c \subseteq \cl B(H)$,
where $v,w \in V$, such that $E_{v,i}F_{w,j} = F_{w,j}E_{v,i}$ satisfy 
\begin{multline}\label{qareldefn}
\forall v, \forall i \ne j, \langle E_{v,i}  F_{v,j} \xi, \xi \rangle < \epsilon,\\
\forall (v,w) \in E, \forall i, \langle E_{v,i}  F_{w,i} \xi, \xi \rangle < \epsilon.
\end{multline}

(vi) The classical \emph{approximate chromatic number} $\chi_{\rm a}(G)$ of $G$ 
is the smallest $c \in \bb N$ such that for every $\epsilon >0$ 
there exists a probability space $(\Omega, P)$ and random variables, 
$f_v: \Omega \to \{ 1,\dots, c \}$ and $ g_w: \Omega \to \{ 1,\dots,c \}$ that satisfy
\begin{multline}\label{qadefnappcn}
\forall v, P( f_v = g_v) > 1- \epsilon,\\
\forall (v,w) \in E,  P(f_v = g_w) < \epsilon.
\end{multline}
\end{defn}

We will now interpret these  definitions in terms of our theory of tensor products of operator systems;
in particular, group operator systems. 
Let $F(n,c) = \bb Z_c * \cdots * \bb Z_c$, where $\bb Z_c = \{0,1,\dots,c-1\}$ 
is the cyclic group with $c$ elements, 
and there are $n$ copies in the free product. 
The C*-algebra of $\bb{Z}_c$ is canonically *-isomorphic to (the abelian C*-algebra) 
$\ell^{\infty}_c = \{(\lambda_i)_{i=1}^c : \lambda_i\in \bb{C}, i = 1,\dots,c\}$. 
Let 
$\cl S(n,c) = \ell^{\infty}_c \oplus_1 \cdots \oplus_1 \ell^{\infty}_c$ 
($n$ copies) be the corresponding operator system coproduct (see, {\it e.g.}, \cite{kavruk2011}). 
By \cite{fkpt}, $\cl S(n,c)$ can be identified with the span of the group $F(n,c)$
inside the C*-algebra $C^*(F(n,c))$.
Let $e_{v,i}$ (resp. $f_{w,j}$) denote the element of 
$\cl S(n,c)$ that is $1$ in the $i$-th component of the $v$-th copy of $\ell^{\infty}_c$
(respectively, $1$ in the $j$-th component of the $w$-th copy) and $0$ elsewhere.
Then 
$$\cl S(n,c) = {\rm span}\{e_{v,i} : v\in V, 1 \le i \le c\} 
= {\rm span}\{ f_{w,j}: v\in V, 1 \le j \le c\}.$$

Given a POVM $(E_{i})_{i=1}^c$ on a Hilbert space $H$, the linear map $\phi : \ell^{\infty}_c \to \cl B(H)$
defined by $\phi(e_i) = E_{i}$, $i = 1,\dots,c$ (where $(e_i)_{i=1}^c$ is 
the canonical basis of $\ell^{\infty}_c$) is unital and completely positive; conversely, 
every unital completely positive map on $\ell^{\infty}_c$ arises in this way. 
It follow from the universal property of the coproduct that 
there exists a bijective correspondence between unital completely positive maps
$\phi : \cl S(n,c) \to \cl B(H)$ and families $(E_{v,i})_{i=1}^c \subseteq \cl B(H)$
of POVM's ($v\in V$), where, given such a family, the corresponding map $\phi$ 
is defined by $\phi( e_{v,i}) = E_{v,i}$, $1\leq i\leq c$, $v\in V$. 

Thus, various proprties for bivariate systems $( E_{v,i})_{i=1}^c$ and $(F_{w,j})_{j=1}^c$ of POVM's  
($v\in V$) can be regarded as arising from different axioms for an 
operator system on the tensor product $\cl S(n,c) \otimes \cl S(n,c).$ 
This viewpoint motivates the following definition.

\begin{defn} \label{d4} Given a functorial operator system tensor product 
$\gamma$ defined for all pairs of operator systems,  
we define the \emph{quantum $\gamma$-chromatic number,}  
$\chi_{{\rm q} \gamma}(G)$ of a graph $G$ to be the smallest $c \in \bb N$ for which there exists  
 a state 
$s: \cl S(n,c) \otimes_{\gamma} \cl S(n,c) \to \bb C$ satisfying 
\begin{multline}\label{cc}
\forall v, \forall i \ne j, s( e_{v,i} \otimes f_{v,j})  =0,\\
\forall (v,w) \in E, \forall i, s (e_{v,i} \otimes f_{w,i} ) =0.
\end{multline}
\end{defn}

\noindent {\bf Remark } 
In Definition \ref{d4}, 
the condition that $s$ be a state can be replaced with the condition that $s$ 
be a non-zero positive linear functional; 
indeed, every such functional can be scaled to be a state without 
effecting the corresponding constraints.

Of particular interest to us will be the following special cases of Definition \ref{d4}:

The \emph{maximal quantum chromatic number} $\chi_{\qmax}(G)$ of $G$, that is,
the smallest $c\in \bb{N}$ for which there exists a state 
$s: \cl S(n,c) \otimes_{\max} \cl S(n,c) \to \bb C$ satisfying conditions (\ref{cc}).

The \emph{minimal quantum chromatic number} $\chi_{\qmin}(G)$ of $G$, that is, 
the smallest $c\in \bb{N}$ for which 
there exists a state $s: \cl S(n,c) \otimes_{\min} \cl S(n,c) \to \bb C$ satisfying (\ref{cc}).

The \emph{commuting quantum chromatic number} 
$\chi_{\qc}(G)$ of $G$, that is, the smallest $c \in \bb N$ for which there exists a state
$s: \cl S(n,c) \otimes_{\rm c} \cl S(n,c) \to \bb C$ satisfying (\ref{cc}).

Before proceeding, we motivate Definition \ref{d4} by placing the classical 
chromatic number in this framework, and then the tensor viewpoint to prove that $\chi(G) = \chi_{\rm a}(G).$
 
To this end, consider the group 
$D(n,c) = \bb Z_c \oplus \cdots \oplus \bb Z_c$ ($n$ copies). 
The C*-algebra $C^*(D(n,c))$ is isomorphic, via Fourier transform, 
to 
$$\ell^{\infty}_c\otimes\cdots \otimes \ell^{\infty}_c\cong \ell^{\infty}(\Delta_{n,c}),$$ 
where $\Delta_{n,c} = \{1,\dots,c\}^n$. 
Letting $\cl S_{\min}(n,c)$ be the operator subsystem of $C^*(D(n,c))$
spanned by the canonical generators of $D(n,c)$, namely, the elements
$$(\delta_{i_1},0,\dots,0), (0,\delta_{i_2},0,\dots,0),\dots,(0,0,\dots,\delta_{i_n}),$$
for $i_k = 1,\dots,c$, $k = 1,\dots,n$ (where $\bb{Z}_c = \{\delta_i : i = 1,\dots,c\}$),
we see that, up to a complete order isomorphism, 
$$\cl S_{\min}(n,c) = {\rm span}\{e'_{v,i} : v\in V, 1 \le i \le c\},$$
where $e'_{v,i}$ is the elementary tensor 
from $\ell^{\infty}_c\otimes\cdots \otimes \ell^{\infty}_c$ having all ones except for the 
$v$-th position, where it has the $i$-th element of the canonical basis of $\ell^{\infty}_c$. 
As above, to improve transparancy, 
we denote by $f'_{w,j}$ the generators of another copy of $\cl S_{\min}(n,c)$.

\begin{proposition}\label{p_relch}
The chromatic number $\chi(G)$ of $G$ is equal to the smallest $c\in \bb{N}$ for which there 
exists a state $s : \cl S_{\min}(n,c) \otimes_{\min} \cl S_{\min}(n,c)\to \bb{C}$ such that 
\begin{multline}\label{mc}
\forall v, \forall i \ne j, s(e'_{v,i} \otimes f'_{v,j})  = 0,\\
\forall (v,w) \in E, \forall i, s(e'_{v,i} \otimes f'_{w,i} ) = 0.
\end{multline}
\end{proposition}
\begin{proof}
We identify $V$ with the set $\{1,2,\dots,n\}$ and
suppose that $s$ is a state of $\cl S_{\min}(n,c) \otimes_{\min} \cl S_{\min}(n,c)$ 
satisfying conditions (\ref{mc}). Then $s$ has an exetension to a state of 
$\ell^{\infty}(\Delta_{n,c}\times\Delta_{n,c})$ satisfying the same conditions. 
Since the elements $e'_{v,i}$ and $f'_{w,j}$ are positive in $\ell^{\infty}(\Delta_{n,c})$, 
there exists a pure state $t$ of $\ell^{\infty}(\Delta_{n,c}\times\Delta_{n,c})$ satisfying the same conditions. 
Thus, there exists a point $(i_1,\dots,i_n,j_1,\dots,j_n)\in \Delta_{n,c}\times\Delta_{n,c}$ 
such that 
$t(u) = u(i_{1},\dots,i_{n},j_1,\dots,j_n)$, $u\in \ell^{\infty}(\Delta_{n,c}\times\Delta_{n,c})$
The first set of conditions (\ref{mc}) imply that $i_v = j_v$, $v\in V$. 
Assign colour $i_v$ to vertex $v$, $v\in V$. 
Then the second set of conditions (\ref{mc}) imply that if $(v,w)\in E$ then 
$i_v\neq i_w$; we have thus obtained a $c$-colouring of $G$ and hence 
$\chi(G)\leq c$. The converse inequality follows by choosing $s$ to be the 
point evaluation at $(i_1,\dots,i_n,i_1,\dots,i_n)$,
where $i_v$ is the colour assigned to $v\in V$.
\end{proof}

\begin{thm} For any graph $G$, we have $\chi_{\rm a}(G) = \chi(G).$
\end{thm}
\begin{proof} Clearly, we have that $\chi_{\rm a}(G) \le \chi(G).$  
Suppose that $c= \chi_{\rm a}(G),$ that $\epsilon >0,$ that $(\Omega, P)$ 
is a probability space, and that $f_v: \Omega \to \{1,\dots,c \},$ 
$g_w : \Omega \to \{ 1,\dots,c \}$ satisfy (\ref{qadefnappcn}). 

Let $L^{\infty}(\Omega, P)$ denote the abelian C*-algebra of bounded measurable 
functions on $\Omega,$ set 
$A_{v,i} = \{ \omega: f_v(\omega) = i\}$ and $B_{w,j} = \{ \omega : g_w(\omega) = j\}$ 
and let $E_{v,i}: \Omega \to \bb C$ and $F_{w,j}: \Omega \to \bb C$ 
denote the characteristic functions of these sets.
Then the linear map 
$\phi_{\epsilon}: \cl S_{\min}(n,c) \otimes_{\min} \cl S_{\min}(n,c) \to L^{\infty}(\Omega, P)$ 
given by $\phi_{\epsilon}(e'_{v,i} \otimes f'_{w,j}) = E_{v,i} F_{w,j}$
is unital and completely positive.
For $u\in \cl S_{\min}(n,c) \otimes_{\min} \cl S_{\min}(n,c)$, 
define 
$$s_{\epsilon}(u) = \int_{\Omega} \phi_{\epsilon}(u)(\omega) dP(\omega); $$
then $s_{\epsilon}$ is a state on $\cl S_{\min}(n,c) \otimes_{\min} \cl S_{\min}(n,c)$.
Taking a limit of a subsequence of states of the form $s_{\epsilon}$ as $\epsilon \to 0$, 
we obtain a state on $\cl S_{\min}(n,c) \otimes_{\min} \cl S_{\min}(n,c)$ 
satisfying the conditions (\ref{mc}).  
By Prosposition \ref{p_relch}, $\chi(G) \le  \chi_{\rm a}(G)$ and the proof is complete.
\end{proof}

\begin{lemma}\label{l_comst}
Let $\cl S$ and $\cl T$ be operator systems and let 
$s: \cl S \otimes_c \cl T \to \bb C$ be a state. Then there exists a Hilbert space $H,$ 
unital completely positive maps, $\phi: \cl S \to B(H),$ $\psi: \cl T \to B(H)$ 
with commuting ranges and a unit vector $\xi \in H,$ such that 
$s(x \otimes y) = \langle \phi(x)\psi(y) \xi, \xi \rangle.$
\end{lemma}
\begin{proof} We have that $\cl S \otimes_c \cl T \subset_{\rm coi} C^*_u(\cl S) \otimes_{\max} C^*_u(\cl T).$ Thus, we may extend the state $s$ to a state, still denoted by $s$, on this C*-algebra.  
The rest follows from considering the GNS representation of the states.
\end{proof}

%\cite[Proposition 10]{jnppsw}

\begin{lemma}\label{l_kas}
Let $\cl A_i$ and $\cl B_j$ be unital C*-algebras, $i = 1,\dots,n$, $j = 1,\dots,m$, and set 
$\cl S = \cl A_1\oplus_1\cdots\oplus_1\cl A_n$, $\cl T = \cl B_1\oplus_1\cdots\oplus_1\cl B_m$,
$\cl A = \cl A_1\ast\cdots\ast\cl A_n$, and 
$\cl B = \cl B_1\ast\cdots\ast\cl B_m$.
Then $\cl S\otimes_{\comm}\cl T\subseteq_{\rm coi} \cl A\otimes_{\max}\cl B$. 
\end{lemma}
\begin{proof}
Let $\nph : \cl S\to \cl B(H)$ and $\psi : \cl T\to \cl B(H)$ be 
unital completely positive maps with commuting ranges. It suffices to show that there exists a  
unital completely positive map $\Gamma : \cl A\ast \cl B\to \cl B(H)$ extending the map 
$\nph\cdot \psi : \cl S\otimes_{\comm}\cl T\to \cl B(H)$ defined by 
$\nph\cdot \psi(x\otimes y) = \nph(x)\psi(y)$.

We identify $\cl S$ with the linear span of $\cl A_1,\dots,\cl A_n$ inside $\cl A$; 
similarly, we identify $\cl T$ with the linear span of $\cl B_1,\dots,\cl B_m$ inside $\cl B$. 
The map $\nph$ is determined by a family $(\nph_i)_{i=1}^n$
of unital completely positive maps (where $\nph_i$
maps $\cl A_i$ into $\cl B(H)$)
via the rule 
$\nph(a_1 + \cdots a_n) = \nph_1(a_1) + \cdots \nph_n(a_n)$, 
$a_i\in \cl A_i$, $i = 1,\dots,n$. 
By \cite{boca}, there exists a unital completely positive map 
$\tilde{\nph} : \cl A\to \cl B(H)$ given by 
$\tilde{\nph}(a_{i_1}\cdots a_{i_k}) = \nph_{i_1}(a_{i_1})\cdots \nph_{i_k}(a_{i_k})$, 
where $a_{i_l}\in \cl A_{i_l}$ for each $l$, and 
$i_1\neq i_2\neq\cdots \neq i_k$.

Similarly, $\psi$ is determined by a family $(\psi_j)_{j=1}^m$, 
where $\psi_j$ is a map from $\cl B_j$ into $\cl B(H)$; moreover, 
$\nph_i(\cl A_i)$ commutes with $\psi_j(\cl B_j)$ for every pair $i,j$ of
indices. Let $\tilde{\psi} : \cl B\to \cl B(H)$ be the unital completely positive map 
given by 
$\tilde{\psi}(b_{j_1}\cdots b_{k_k}) = \psi_{j_1}(b_{j_1})\cdots \psi_{j_k}(b_{j_k})$, 
where $b_{j_l}\in \cl B_{j_l}$ for each $l$, and 
$j_1\neq j_2\neq\cdots \neq j_k$. Since the linear span of 
free words are dense in the corresponding free products $\cl A$ and $\cl B$, we 
have that $\tilde{\nph}(\cl A)$ and $\tilde{\psi}(\cl B)$ commute. 
It is now clear that the unital completely positive 
map $\Gamma : \cl A\ast\cl B\to \cl B(H)$
arising from $\tilde{\nph}$ and $\tilde{\psi}$ in a similar fashion \cite{boca} extends $\nph\cdot\psi$.
\end{proof}

\begin{theorem} 
Let $G$ be a graph. Then $\chi_{\qc}(G) = \chi_{\rm qar}(G) = \chi_{\rm qr}(G).$

Moreover, $\chi_{\qc}(G)\leq c$ if and only if there exist a Hilbert space $H$,
a unit vector $\xi\in H$ and PVM's
$(E_{v,i})_{i=1}^c$ and $(F_{w,j})_{j=1}^c$ on $H$, $v,w\in V$, such that
conditions (\ref{qcreldefn}) hold. 
%\begin{multline}\label{eq_minep}
%\forall v, \forall i \ne j, \langle (E_{v,i}F_{v,j}) \xi, \xi \rangle =0,\\
%\forall (v,w) \in E, \forall i, \langle (E_{v,i}F_{w,i}) \xi, \xi \rangle  = 0.
%\end{multline}
\end{theorem}
\begin{proof}  
Clearly, $\chi_{\rm qar}(G) \le \chi_{\rm qr}(G).$ 
If $\chi_{\rm qar}(G) =c$ then for every $\epsilon >0$ there is a system 
of POVM's on a Hilbert space and vectors satisfying (\ref{qareldefn}). 
But each such system defines a state $s_{\epsilon}: \cl S(n,c) \otimes_{\rm c} \cl S(n,c) \to \bb C$ by setting
\[ s_{\epsilon}(e_{v,i} \otimes f_{w,j}) = \langle E_{v,i}F_{w,j} \xi, \xi \rangle.\]

By taking a limit point of these states as $\epsilon \to 0,$ we obtain a state on $\cl S(n,c) \otimes_c \cl S(n,c)$ satisfying (\ref{cc}). This proves that $\chi_{\rm qc}(G) \le c.$

Finally, if $c= \chi_{\rm qc}(G),$ then there exists a state 
$s: \cl S(n,c) \otimes_{\rm c} \cl S(n,c) \to \bb C$ satisfying (\ref{cc}). 
By Lemma \ref{l_comst}, there exists a Hilbert space $H,$ a pair of unital, 
completely positive maps $\phi$ and $\psi$ with commuting ranges and a unit vector $\xi$ such that
\[s(e_{v,i} \otimes f_{w,j}) = \langle \phi(e_{v,i}) \psi(f_{w,j}) \xi, \xi \rangle .\]
However, setting $E_{v,i} = \phi(e_{v,i})$ and $F_{w,j} = \psi(f_{w,j})$ yields a set of POVM's and a vector satisfying (\ref{qcreldefn}). It follows that $\chi_{\rm qr}(G) \le c$. 

For the last statement, note that, by Lemma \ref{l_kas}, every state $s$ of $\cl S(n,c) \otimes_{\rm c} \cl S(n,c)$
extends to a state $t$ of $C^*(F(n,c)) \otimes_{\max} C^*(F(n,c))$. 
The GNS representation of $t$ yields *-representations $\pi$ and $\rho$ of $C^*(F(n,c))$
on a Hilbert space $H$, with commuting ranges, 
and a vector $\xi\in H$ such that 
$t(a\otimes b) = \langle \pi(a)\rho(b)\xi,\xi\rangle$, $a,b\in C^*(F(n,c))$. Since $e_{v,i}$
and $f_{w,j}$ are projections, we have that $(\pi(e_{v,i}))_{i=1}^c$, 
$(\rho(f_{w,j}))_{j=1}^c$ are families of PVM's with the desired properties. 
\end{proof}

We now wish to turn our attention to the various spacial versions.

\begin{theorem}\label{th_qmin}
Let $G = (V,E)$ be a graph on $n$ vertices. 
The following are equivalent:

(i) \ \ $\chi_{\qmin}(G) \le c$;

(ii) \ there exists a state 
$s: C^*(F(n,c)) \otimes_{\min} C^*(F(n,c)) \to \bb C$ satisfying conditions (\ref{cc});

(iii) for every $\epsilon > 0$, there exist Hilbert spaces $H$ and $K$,
a unit vector $\xi\in H\otimes K$ and POVM's
$(E_{v,i})_{i=1}^c \subseteq \cl B(H)$ and $(F_{w,j})_{j=1}^c \subseteq \cl B(K)$,
where $v,w \in V$, such that
\begin{multline}\label{eq_minep1}
\forall v, \forall i \ne j, \langle (E_{v,i}\otimes F_{v,j}) \xi, \xi \rangle < \epsilon,\\
\forall (v,w) \in E, \forall i, \langle (E_{v,i} \otimes F_{w,i}) \xi, \xi \rangle  < \epsilon.
\end{multline}

(iv) 
for every $\epsilon > 0$, there exist $p,q\in \bb{N}$, 
a unit vector $\xi\in \bb{C}^p\otimes \bb{C}^q$ and POVM's
$(E_{v,i})_{i=1}^c \subseteq M_p$ and $(F_{w,j})_{j=1}^c \subseteq M_q$,
where $v,w \in V$, such that (\ref{eq_minep1}) hold.
%\begin{multline}\label{eq_minfin}
%\forall v, \forall i \ne j, \langle (E_{v,i}\otimes F_{v,j}) \xi, \xi \rangle < \epsilon,\\
%\forall (v,w) \in E, \forall i, \langle (E_{v,i} \otimes F_{w,i}) \xi, \xi \rangle  < \epsilon. 
%\end{multline}

Consequently, $\chi_{\rm qmin}(G) = \chi_{\rm qas}(G) = \chi_{\rm qa}(G).$
\end{theorem}
\begin{proof}
(i)$\Leftrightarrow$(ii) follows from Krein's Theorem and the fact that 
$\cl S(n,c)\otimes_{\min}\cl S(n,c)\coisubset C^*(F(n,c)) \otimes_{\min} C^*(F(n,c))$.

(ii)$\Rightarrow$(iii) Let $\pi : C^*(F(n,c))\to \cl B(H)$ be a faithful representation of $C^*(F(n,c))$, where
$H$ is a Hilbert space. Then 
$\pi\otimes\pi : C^*(F(n,c)) \otimes_{\min} C^*(F(n,c))\to \cl B(H\otimes H)$ is a faithful 
representation. Let $E_{v,\alpha} = \pi(e_{v,\alpha})$ and $F_{w,\beta}' = \pi(e_{w,\beta})$. 
Let $s$ be a state satisfying (ii) and $\epsilon > 0$. 
By \cite[Corollary 4.3.10]{krII}, 
$s$ belongs to the weak* closure of the convex hull of the set of all vector states of 
$(\pi\otimes\pi)(C^*(F(n,c)) \otimes_{\min} C^*(F(n,c)))$. Thus, 
there exist vectors $\xi_1,\dots,\xi_r\in  H\otimes H$ and positive scalars $\lambda_1,\dots,\lambda_r$
with $\sum_{l=1}^r \lambda_l = 1$
such that 
\begin{equation}\label{eq_app}
\left|s(e_{v,i}\otimes f_{w,j}) - \sum_{l=1}^r \lambda_l \langle (E_{v,i}\otimes F_{w,j}') \xi_l, \xi_l \rangle\right| < \epsilon,
\end{equation}
for all $v,w\in V$ and all $i,j = 1,\dots,c$.
Let $K = H\otimes\bb{C}^r$ and $\xi = \sum_{l=1}^r (\sqrt{\lambda_l}\xi_l)\otimes \delta_l$, 
where $(\delta_l)_{l=1}^r$ is the canonical basis of $\bb{C}^r$. 
Then $\|\xi\| = 1$. Let $F_{w,j} = F_{w,j}'\otimes 1_r$, $w\in V$, $j = 1,\dots,c$. 
Then $(F_{w,j})_{j = 1}^c$ is a POVM on the Hilbert space $K$ for each $w\in V$, and 
(\ref{eq_app}) can be written as 
$$\left|s(e_{v,i}\otimes f_{w,j}) - \langle (E_{v,i}\otimes F_{w,j}) \xi, \xi \rangle\right| < \epsilon, \ \ v,w\in V, i,j = 1,\dots,c.$$
Conditions (\ref{cc}) now imply that relations (\ref{eq_minep1}) are satisfied. 

(iii)$\Rightarrow$(iv) 
For a given $\epsilon > 0$, let $H$ and $K$ be Hilbert spaces, 
$\xi\in H\otimes K$ be a unit vector and 
$(E_{v,i})_{i=1}^c \subseteq \cl B(H)$ and $(F_{w,j})_{j=1}^c \subseteq \cl B(K)$ be POVM's
such that (\ref{eq_minep1}) are satisfied for $\epsilon /2$ in the place of $\epsilon$. 
We may assume that $H$ and $K$ are separable since the vector $\xi$ 
has at most countably many non-zero coefficients with respect to a given orthonormal basis.
Let $(P_m)_{m\in \bb{N}}$ (resp. $(Q_m)_{m\in \bb{N}}$) be an increasing sequence of 
projections of finite rank on $H$ (resp. $K$), strongly convergent to the identity operator. 
Let $E_{v,i}^m = P_mE_{v,i}P_m$ and $F_{w,j}^m = Q_mF_{w,j}Q_m$; 
then the families $(E_{v,i}^m)_{i=1}^c$ and $(F_{w,j}^m)_{j=1}^c$
are POVM's on the Hilbert spaces $P_mH$ and $Q_mK$, respectively.
Since $(P_m\otimes Q_m)\xi\to_{m\to\infty} \xi$, 
there exists $k$ such that (\ref{eq_minep1}) are satisfied for 
$(E_{v,i}^k)_{i=1}^c$ and $(F_{w,j}^k)_{j=1}^c$.

(iv)$\Rightarrow$(ii) Let $\epsilon > 0$. A choice of POVM's as in (iv) gives rise, as in the proof of 
Lemma \ref{l_kas}, to unital completely positive maps 
$\pi$ and $\rho$ of $C^*(F(n,c))$ on $\bb{C}^p$ and $\bb{C}^q$, 
respectively. Let $\pi\otimes_{\min}\rho$ be the corresponding unital completely positive map from
$C^*(F(n,c))\otimes_{\min}C^*(F(n,c))$ into $\cl B(\bb{C}^p\otimes \bb{C}^q)$ and let 
$s_{\epsilon}$ be the state of $C^*(F(n,c))\otimes_{\min}C^*(F(n,c))$ given by 
$s_{\epsilon}(u) = \langle (\pi\otimes_{\min}\rho)(u)\xi,\xi\rangle$, $u\in C^*(F(n,c))\otimes_{\min}C^*(F(n,c))$. 
Let $s$ be a weak* cluster point of the set $\{s_{1/m}\}_{m\in \bb{N}}$. 
Then $s$ satisfies (\ref{cc}). 
\end{proof}

The following proposition describes the relations between the 
introduced parameters.

\begin{prop}\label{p_ineq}
Let $G$ be a graph on $n$ vertices.  Then
\[\chi_{\qmax}(G) \le \chi_{\qc}(G) \le \chi_{\qmin}(G) \le \chi_{\qs}(G) \le \chi_{\q}(G)\le \chi(G).\]
\end{prop}
\begin{proof}  
The canonical maps 
$$\cl S(n,c)\otimes_{\max}\cl S(n,c)\longrightarrow \cl S(n,c)\otimes_{\comm}\cl S(n,c)
\longrightarrow \cl S(n,c)\otimes_{\min}\cl S(n,c)$$ are completely positive; thus, 
every state on $\cl S(n,c)\otimes_{\min}\cl S(n,c)$ 
(resp. $\cl S(n,c)\otimes_{\comm}\cl S(n,c)$) 
is also a state
when considered as a map on $\cl S(n,c)\otimes_{\comm}\cl S(n,c)$
(resp. $\cl S(n,c)\otimes_{\max}\cl S(n,c)$). 
It follows that
$\chi_{\qmax}(G) \le \chi_{\qc}(G) \le \chi_{\qmin}(G)$.
The inequalities $\chi_{\qs}(G) \le \chi_{\q}(G)\le \chi(G)$ are trivial. 
The inequality $\chi_{\qmin}(G) \le \chi_{\qs}(G)$ follows from Theorem \ref{th_qmin}. 
\end{proof}

\begin{corollary}\label{c_kc}
If the Kirchberg Conjecture holds true then $\chi_{\qc}(G) = \chi_{\qmin}(G)$
for every graph $G$. 
\end{corollary}
\begin{proof} 
The validity of the  Kirchberg Conjecture implies \cite{fritz} that
$$C^*(F(n,c)) \otimes_{\min} C^*(F(n,c)) = C^*(F(n,c)) \otimes_{\max} C^*(F(n,c));$$
by Lemma \ref{l_kas}, this implies that 
$\cl S(n,c) \otimes_{\min} \cl S(n,c) = \cl S(n,c) \otimes_{\rm c}\cl S(n,c)$. 
Thus, $\chi_{\qmin}(G) = \chi_{\qc}(G)$.
\end{proof}

\section{Connections with non-signalling boxes}\label{s_nsb}

In this section, we express the quantum chromatic numbers introduced above in dual terms. 
This will allow us to exhibit a connection between the quantities we introduced and 
non-signalling boxes arising from quantum correlations. 
Recall first \cite{ce} that, given a finite dimensional operator system $\cl S$, 
its normed space dual can be equipped with a natural matricial ordering. 
Moreover, any faithful state on $\cl S$ is an Archimedean order unit for this ordering;
the resulting operator system is denoted by $\cl S^d$ (note that a specific choice 
of a faithful state is always made in advance). 

We will need the following result \cite[Theorem 5.9]{fkpt}:

\begin{theorem}\label{th_sncdual}
Let $\cl V(n,c)$ be the operator system dual of $\cl S(n,c)$ and 
$\sigma: \ell^{\infty}_c \to \bb C$ be given by $\sigma(( x_1,\dots, x_c)) = \sum_{i=1}^c x_i$. 
Then 
$$\cl V(n,c) \cong \{(x_1,\dots,x_n) : x_i\in \ell_c^{\infty}, \sigma(x_i) = \sigma(x_j), i,j = 1,\dots,n\}
\coisubset \ell_{nc}^{\infty},$$
up to a complete order isomorphism.
\end{theorem}

The duality between $\cl V(n,c)$ and $\cl S(n,c)$ is given as follows:
if $$u = \sum_{v,i} \lambda_{v,i}e_{v,i}\in \cl S(n,c)$$ and 
$f = ((x_{v,i})_{i=1}^c)_{v\in V}$, then $\langle u,f\rangle = \sum_{v,i} \lambda_{v,i} x_{v,i}$. 

For $k\in \bb{N}$, let 
$\rho_k : \ell_k^{\infty} \otimes \ell_k^{\infty} \to M_k$ be the mapping given by 
$$\rho((\lambda_1,\dots,\lambda_{k})\otimes (\mu_1,\dots,\mu_{k})) 
= (\lambda_i\mu_j)_{i,j=1}^{k}.$$
The mapping $\rho_k$ sends the algebraic tensor product $\ell_{nc}\otimes\ell_{nc}$
onto $M_{nc}$.

For a matrix $A\in M_c$, let $R_i(A)$ (resp. $C_i(A)$) denote its $i$th row (resp. column). 
Let 
$$\cl M(n,c) \stackrel{def}{=} \{(A_{i,j})_{i,j=1}^n\in M_n(M_c) : 
\sigma(R_p(A_{i,j})) = \sigma(R_p(A_{i,k}))$$
$$\mbox{ and } 
\sigma(C_p(A_{i,j})) = \sigma(C_p(A_{k,j})), i,j,k = 1,\dots,n, p = 1,\dots,c\}.$$
Note that 
$$\rho_{nc}(\cl V(n,c)\otimes \cl V(n,c)) = \cl M(n,c).$$ 
Indeed, the inclusion of $\rho_{nc}(\cl V(n,c)\otimes \cl V(n,c))$ into the right hand side 
follows directly from the definition of $\rho_{nc}$. 
Conversely, let $A\in \cl M(n,c)$. Letting $\delta_s$ be the evaluation functional 
on $\ell_{nc}^{\infty}$ on the $s$-th coordinate, we see that 
$\id\otimes\delta_s$ and $\delta_s\otimes \id$ map $\rho^{-1}(A)$ into $\cl V(n,c)$. 
This shows that $\rho^{-1}(A)\in \cl V(n,c)\otimes\cl V(n,c)$, and the claim is proved.
We have 
$\dim(\cl S(n,c)) = \dim(\cl V(n,c)) = nc -(n-1) = n(c-1) +1$, 
so that $\dim(\cl V(n,c) \otimes \cl V(n,c)) = \dim\cl M(n,c) = n^2(c-1)^2 + 2n(c-1) +1.$
We note that the spaces $\cl M(n,c)$ are higher dimensional versions of the 
space of non-signaling boxes, see \cite{fkpt_dg}.

We can now express some of the quantum chromatic numbers 
introduced in Section \ref{s_qcn} in terms of the space $\cl M(n,c)$. 
Set 
$$\cl M(n,c)_{\min}^+ = \rho_{nc}((\cl V(n,c)\otimes_{\min} \cl V(n,c))^+)$$
and 
$$\cl M(n,c)_{\max}^+ = \rho_{nc}((\cl V(n,c)\otimes_{\max} \cl V(n,c))^+).$$
Note that 
$$\cl V(n,c)\otimes_{\min}\cl V(n,c)\coisubset \ell_{(nc)^2}$$
and so $\cl M(n,c)_{\min}^+$ consists of all matrices in $\cl M(n,c)$ with non-negative entries. 
Also note that, by \cite[Proposition 1.9]{fp}, 
$$(\cl S(n,c) \otimes_{\max}\cl S(n,c))^d = \cl V(n,c)\otimes_{\min}\cl V(n,c)$$
and 
$$(\cl S(n,c) \otimes_{\min}\cl S(n,c))^d = \cl V(n,c)\otimes_{\max}\cl V(n,c).$$

\begin{lemma}\label{l_sind}
Let $G = (V,E)$ be a graph.
Let $s\in \cl V(n,c)\otimes_{\min}\cl V(n,c)$ be a state on $\cl S(n,c) \otimes_{\max}\cl S(n,c)$
and write $\rho_{nc}(s) = (A_{v,w})_{v,w\in V}$, where $A_{v,w}\in M_c$, $v,w\in V$. 
Then $s$ satisfies conditions (\ref{cc}) if and only if $A_{v,v}$ is a diagonal matrix for each $v\in V$
and $A_{v,w}$ has a zero diagonal whenever $(v,w)\in E$. 
\end{lemma}
\begin{proof}
The claim follows from the fact that the element of $\cl M(n,c)$ corresponding to a state 
$s$ on $\cl S(n,c) \otimes_{\max}\cl S(n,c)$ is $((s(e_{v,i}\otimes f_{w,j})_{i,j=1}^c)_{v,w\in V}$.
By duality, these matrices are the coefficients of the positive elements of $\cl V(n,c) \otimes_{\min} \cl V(n,c).$  The set of matrices that can be obtained in this fashion is just the set of all matrices with non-negative entries that satisfy the linear constraints given by the requirement that $\sigma(x_i) = \sigma(x_j).$ The result follows by applying the reasoning in \cite[Theorem~7.1]{fkpt_dg} which shows that in the case $c=2$ the matrices that one can obtain in this fashion are, up to scaling, the matrices that are non-signalling boxes.
\end{proof}

\begin{proposition}\label{p_qmaxc}
For every graph $G$ on more than one vertices, we have $\chi_{\qmax}(G) = 2$.
\end{proposition}
\begin{proof}
Set $A_{v,v} = I_2$, $A_{v,w} = H := \begin{pmatrix} 0&1\\1&0 \end{pmatrix}$ if $v \ne w$, and 
$A = (A_{v,w})_{v,w\in V}$. Then $A\in \cl M(n,c)_{\min}^+$ and, by 
Lemma \ref{l_sind}, $\rho_{nc}^{-1}(A)$ is a state satisfying (\ref{cc}). It follows that $\chi_{\qmax}(G)\leq 2$.
The fact that $\chi_{\qmax} > 1$ if $|V| > 1$ follows from Lemma \ref{l_sind}.
\end{proof}

\section{The vectorial chromatic number}

In this section we introduce another parameter of a graph, which we call 
the vectorial chromatic number, and investigate its relations with the quantum chromatic numbers
defined in Section \ref{s_qcn}. The vectorial chromatic number will prove useful in 
supplementing the inequalities established in Proposition \ref{p_ineq} by an important 
lower bound. 

After writing this section, we learned of the recent preprint of \cite{cmrssw}, which includes a parameter that is equivalent to our vectorial chromatic number.  Many of our results can be derived from their results, yet our proofs are generally quite different.

\begin{defn}\label{d_vec}
Let $G=(V,E)$ be a  graph and $c\in \bb{N}$.  A \emph{vectorial $c$-colouring} of $G$
is a set of vectors $\{x_{v,i} : v \in V, 1 \le i \le c\}$ in a Hilbert space such that
$$\langle x_{v,i}, x_{w,j} \rangle \ge 0, \ \ \ v,w\in V, 1\leq i,j\leq c,$$
$$\sum_{i=1}^c x_{v,i} = \sum_{i=1}^c x_{w,i}, \ \ \ 
\left\| \sum_{i=1}^c x_{v,i} \right\| = 1, \ \ \  v,w \in V,$$
$$\langle x_{v,i}, x_{v,j} \rangle = 0,  \ \ \ v\in V, i \ne j,$$
$$\langle x_{v,i},x_{w,i} \rangle  = 0, \ \ \ (v,w) \in E, 1\leq i\leq c.$$
The least integer $c$ for which there exists a vectorial $c$-colouring, will be denoted $\chi_{\rm vect}(G)$ 
and called the \emph{vectorial chromatic number} of G.
\end{defn}

\begin{remark} Our vectorial chromatic number is not equal to the chromatic number $\chi_{\rm vec}(G)$ studied in \cite{SS12}. In fact, if $C_5$ denotes the 5-cycle then $\chi_{\rm vec}(C_5) \le \vartheta(C_5) = \sqrt{5},$ while it follows from Corollary~\ref{chrom=3}, that $\chi_{\rm vect}(C_5) = \chi(C_5) =3.$
\end{remark}

\begin{remark} Let $G$ and $H$ be graphs. In \cite[Definition~3]{cmrssw}, the concept of  $G \stackrel{+}{\to}  H$ is defined.  If $K_c$ denotes the complete graph on $c$ vertices, then it follows that $\chi_{\rm vect}(G) =c$ if and only if $G \stackrel{+}{\to} K_c.$
\end{remark} 

The vectorial chromatic  number of a graph $G$ is a relaxed geometric version of the 
relativistic quantum chromatic number of $G$, as can be seen from the proof 
of the next proposition.  In \cite{cmrssw}, this result follows from the
facts that $G \stackrel{q}{\to} K_c \implies G \stackrel{+}{\to} K_c.$

\begin{prop}\label{p_vec}
Let $G=(V,E)$ be a graph.  Then $\chi_{\vecc}(G) \le \chi_{\qc}(G).$
\end{prop}
\begin{proof}
Let $\cl H$ be a Hilbert space, $\xi$ be a unit vector and 
$( E_{v,i})_{i=1}^c$ and $(F_{w,j})_{j=1}^c$ be POVM's satisfying 
the conditions of Definition \ref{d_3} (iv). 
Set $x_{v,i} = E_{v,i} \xi$ and $y_{w,j} = F_{w,j} \xi.$

Note that, since $E_{v,i}F_{w,j} = F_{w,j}E_{v,i}$ and $E_{v,i}\geq 0$, $F_{w,j}\geq 0$, 
we have that
\[ \langle x_{v,i}, y_{w,j} \rangle = \langle y_{w,j}, x_{v,i} \rangle \ge 0.\]

Since $\sum_{i=1}^c E_{v,i} =I,$ we have that $\sum_{i=1}^c x_{v,i} = \xi.$
Moreover, $\langle x_{v,i} , y_{v,j} \rangle =0$, $ i \ne j.$  Hence,
\begin{eqnarray*}
1 = \langle\xi,\xi\rangle  = \sum_{i,j=1}^c \langle x_{v,i}, y_{v,j} \rangle
& = & \sum_{j=1}^c \langle x_{v,j}, y_{v,j} \rangle\\
& \le & \left(\sum_{j=1}^c \|x_{v,j}\|^2 \right)^{1/2} 
\left(\sum_{j=1}^c \|y_{v,j}\|^2\right)^{1/2} .
\end{eqnarray*}
However,
\[ \sum_{j=1}^c \|x_{v,j}\|^2 = \sum_{j=1}^c \langle E_{v,j}^2 \xi, \xi \rangle \le \sum_{j=1}^c \langle E_{v,j} \xi, \xi \rangle= \langle \xi, \xi \rangle =1,\]
and, similarly, $\sum_{j=1}^c \|y_{w,j}\|^2\leq 1$.
Hence, we must have equality throughout, so $y_{v,j} = \lambda_{v,j} x_{v,j}$
for some unimodular scalar $\lambda_{v,j}$. The fact that 
$\langle x_{v,i}, y_{w,j} \rangle \geq 0$ implies that $\lambda_{v,j} = 1$ for all $v$ and $j$.
Thus, 
\[ x_{v,j} = y_{v,j} \text{ and } x_{v,i} \perp x_{v,j} \mbox{ if } i \ne j.\]
%Since $\sum_{j=1}^c \|x_{v,j}\|^2 =1$, $v\in V$, we have that 
%the vectors $x_{v,i}$, $i=1,\dots,c$, are orthogonal, for each $v\in V$. 

Finally, note that if $(v,w)\in E$ then
\[\langle x_{v,i}, x_{w,i} \rangle = \langle x_{v,i}, y_{w,i} \rangle = \langle E_{v,i}F_{w,i}\xi, \xi\rangle = 0.\]
It now follows that the family $\{x_{v,i} : 1\leq i \leq c, v\in V\}$ is a vectorial $c$-colouring of $G$, and 
hence $\chi_{\rm vec}(G)\leq \chi_{\rm qc}(G)$.
\end{proof}

The vectorial chromatic number is sometimes simpler to deal with than the quantum chromatic number.  For example, in \cite{AHKS06} a quantum protocol is created to show that  the Hadamard graph $\Omega_N$,  has a quantum 
$N$-colouring, {\it i.e.}, that $\chi_q(\Omega_N) \le N.$ Recall that $\Omega_N$ is the graph on the $2^N$ vertices, which  are identified with all  $N$-tuples $v=( v(0),...,v(N-1))$ with $v(i) \in \{-1, +1 \}$ and vertex $v$ connected to vertex $w$ if and only if $\langle v,w \rangle =0.$ 

 Since $\chi_{\rm vect}(G) \le \chi_{\rm q}(G),$ showing that $\chi_{\rm vect}(\Omega_N) \le N,$ is a weaker result. The following applies to a broad family of orthogonality graphs.

\begin{prop}\label{orthograph} Let $V \subseteq \bb C^N$ be a set of $M$ vectors such that  $v=( v(0),...,v(N-1)) \in V$ implies that $|v(i)|=1$ and let $G=(V,E)$ be the graph defined by $(v,w) \in E$ if and only if $\langle v,w \rangle =0.$ Then $\chi_{\rm vect}(G) \le N.$
\end{prop}
\begin{proof}
Let $\omega= e^{2 \pi i/N},$ and for $0 \le l \le N-1$ define vectors $x_{v,l} \in \bb C^N \otimes \bb C^N$ by
\[ x_{v,l} = \frac{1}{N^{3/2}} \sum_{i,j=0}^{N-1} v(i)\overline{v(j)} \omega^{(i-j)l} e_i \otimes e_j ,\]
then it is easy to check that these vectors define a vectorial $N$-colouring of $G.$
\end{proof} 

By Propositions \ref{p_ineq} and \ref{p_vec}, 
$\chi_{\rm vect}(G) \le \chi_{\qc}(G) \le \chi_{\rm q}(G)$; hence, the next result is a slight
improvement of \cite[Proposition 3]{cnmsw}.

\begin{thm}\label{th_vec2} 
Let $G$ be a connected graph.  Then $\chi_{\rm vect}(G)=2$ if and only if $\chi(G) =2.$
\end{thm}
\begin{proof} 
Assume that $\chi_{\rm vect}(G) = 2$ and fix a 
vectorial $2$-colouring $\{x_{v,i} : i = 1,2, v\in V\}$ of $G$. 
%Thus, for each $v$ $\|x_{v,1}\|^2 + \|x_{v,2}\|^2 =1$ with $x_{v_1} \perp x_{v,2}.$

If $(v,w)$ is an edge, then
\[ \|x_{v,1}\|^2 = \langle x_{v,1}, \xi \rangle = \sum_{j=1}^2 \langle x_{v,1}, x_{w,j} \rangle = \langle x_{v,1}, x_{w,2} \rangle \le \|x_{v,1}\| \|x_{w,2}\|.\] 
Hence, $\|x_{v,1}\| \le \|x_{w,2}\|,$ but repeating the argument beginning with $x_{w,2},$ we obtain,  $\|x_{v,1}\|= \|x_{w,2}\|$ and hence, all the above inequalities are equalities.
Thus, if $(v,w)$ is an edge, then 
$x_{v,1} = x_{w,2}$ and, by symmetry, $x_{v,2} = x_{w,1}.$  
Fix a vertex $v_0$ and using that the graph is connected we see that for every $w$ there is a unique $i$ so that $x_{w,i} = x_{v_0,1},$ in which case we will assign $w$ colour $i.$  Any vertex $z$ that is adjacent to $w$ will have
$x_{w,i} = x_{z, i+1},$ where the arithmetic is mod $2$.
We have thus defined a $2$-colouring of $G.$

The converse follows from the fact that for any graph with more than one vertex, 
$2 \le \chi_{\rm vect}(G) \le \chi(G).$  Hence, when $\chi(G)=2$, we have equality. 
\end{proof}

\begin{corollary}\label{chrom=3}
If $\chi(G)=3$ then $\chi_{\rm vect}(G) =3$. 
In particular, if $C_5$ is the $5$-cycle then $\chi_{\rm vect}(C_5) = 3$.
\end{corollary}
\begin{proof}
Suppose that $\chi(G)=3$; then $\chi_{\rm vect}(G)\leq 3$. However,
if $\chi_{\rm vect}(G) = 2$ then by Theorem \ref{th_vec2}, $\chi(G) =2$, hence 
$\chi_{\rm vect}(G) \neq 2$. It is easy to see that $\chi_{\rm vect}(G) = 1$
only for the graph on one vertex. Thus, $\chi_{\rm vect}(G) = 3$. 

The claim for the graph $C_5$ follows from the fact that $\chi(C_5) = 3$.
\end{proof}

Schrijver defined a lower bound $\vartheta^-(G)$ for the Lovasz theta function, $\vartheta(G),$ which we use in the theorem below. Our original inequality was weaker in that it used $\vartheta(G).$  We are grateful to the authors of \cite{cmrssw} for pointing out this strengthening.  This result can also be deduced 
from other inequalities found in their paper.

\begin{theorem}\label{p_fullvec}
Let $G$ be a graph on $n$ vertices. 
Then $\chi_{\vecc}(G)\vartheta^-(G) \geq n$. 
In particular, if $G$ is the complete graph on $n$ vertices, then $\chi_{\vecc}(G) = n$.
\end{theorem}
\begin{proof}
Let $(x_{v,j})_{j=1}^c$, $v\in V$, be families of vectors in a Hilbert space $H$ 
satisfying the conditions of Definition \ref{d_vec}.  
For $i,j = 1,\dots,c$, let 
$Q_{i,j} = (\langle x_{v,i},x_{w,j}\rangle)_{v,w\in V}$; thus, $Q_{i,j}\in M_n$.
Set $Q = (Q_{i,j})_{i,j=1}^c\in M_c(M_n) = M_{nc}$. 
Note that $Q$ is the Grammian of the family $\{x_{v,i} : v\in V,1\leq i\leq c\}$
and is hence positive. By Choi's Theorem, 
the linear map $\Phi : M_c\to M_n$ defined by $\Phi(E_{i,j}) = Q_{i,j}$, 
is completely positive.
Note that 
$$\Phi(I) = \sum_{i=1}^c \Phi(E_{i,i}) = \sum_{i=1}^c Q_{i,i}.$$
The $(v,v)$-entry of $\Phi(I)$ is thus 
$$\sum_{i=1}^c \langle x_{v,i},x_{v,i}\rangle = 1.$$
On the other hand, if $(v,w)\in E$ then $\langle x_{v,i},x_{w,i}\rangle = 0$ and thus 
the $(v,w)$-entry of $\Phi(I)$ is zero. It follows that $\Phi(I) = I + K$, where 
$K$ is supported on the set of edges of the complement of $G$.
Moreover, since all the entries of $K$ are non-negative, by \cite[Theorem~26]{cmrssw}, we have that $\|I+K\| \le \vartheta^-(G).$ 
Thus, $\|\Phi\| = \|\Phi(I)\|\leq \theta^-(G)$. 

On the other hand, letting $J_k$ denote the matrix in $M_k$ ($k\in \bb{N}$)
whose each entry is equal to $1$, we have that 
$\Phi(J_c) = \sum_{i,j=1}^c Q_{i,j}$. 
The $(v,w)$-entry of $\Phi(J_c)$ is this equal to 
$$\sum_{i,j=1}^c \langle x_{v,i},x_{w,j}\rangle =
\left\langle \sum_{i=1}^c x_{v,i}, \sum_{j=1}^c x_{w,j}\right\rangle = \|\xi\|^2 = 1.$$
In other words, $\Phi(J_c) = J_n$. 
Thus, 
$$n = \|J_n\| = \|\Phi(J_c)\|\leq \vartheta^-(G)\|J_c\| = c\vartheta^-(G).$$
The inequality now follows after choosing a vectorial $\chi_{\rm vect}(G)$-colouring.

If $G$ is the complete graph on $n$ vertices then its complement is the null graph on $n$ vertices,
and hence $\vartheta(G) = 1$. Thus, the inequality established in the previous paragraph 
reduces to $\chi_{\vecc}(G) \geq n$. Since $\chi_{\rm vect}(G)\leq \chi(G)\leq n$, 
we have that $\chi_{\rm vect}(G) = n$. 
\end{proof}

\begin{remark} For the orthogonality graphs on $M$ vertices of Proposition~\ref{orthograph} we see that $\vartheta^-(G) \ge \frac{M}{N}.$
\end{remark}

%\begin{proposition}\label{p_gho}
%Let $G$ and $H$ be graphs and $\nph : G \to H$ be a homomorphism. If 
%$\chi_{\diamond} \in \{\vecc, \qc, \qmin, \qs, {\rm q}\}$, 
%then $\chi_{\diamond}(G)\leq \chi_{\diamond}(H)$. 
%\end{proposition}
%\begin{proof}
%We give the proof for $\diamond = \qc$; the rest of the cases are similar.
%Let  $(E_{v,i})_{i=1}^c$ and $(F_{w,j})_{j=1}^c$ ($v,w\in V(H)$) 
%be POVM's satisfying the conditions of Definition \ref{d_3} (iv), for the graph $H$. 
%Then $(E_{\nph(v'),i})_{i=1}^c$ and $(F_{\nph(w'),j})_{j=1}^c$ ($v',w'\in V(G)$)
%satisfy the same conditions for the graph $G$. Indeed,  
%if $i\neq j$ then $\langle E_{\nph(v'),i}F_{\nph(v'),j}\xi,\xi\rangle = 0$.
%On the other hand, if 
%$(v',w')\in E(G)$ then $(\nph(v'),\nph(w'))\in E(H)$ and hence 
%$\langle E_{\nph(v'),i}F_{\nph(w'),i}\xi,\xi\rangle = 0$. 
%\end{proof}

For a given graph $G$, Szegedy's $\vartheta^+(G)$ is a quantity that satisfies $\vartheta(G) \le \vartheta^+(G).$ See \cite{cmrssw} for the definition. The following inequality follows by applying the results from 
\cite{cmrssw}.

\begin{corollary}
For any graph $G$, we have 
$$ \vartheta^+(\overline{G})  \leq \chi_{\rm vect}(G)
\le \chi_{\qc}(G) \le \chi_{\qmin}(G) \le \chi_{\qs}(G) \le \chi_{\q}(G)\le \chi(G).$$
\end{corollary}
\begin{proof}
%Let $k = \omega(G)$; then $G$ has an induced full subgraph $H$ of size $k$.
%The embedding $H\to G$ is a homomorphism; thus, by Theorem \ref{p_fullvec} and 
%Proposition \ref{p_gho}, 
%$$k = \chi_{\rm vect}(H)\leq \chi_{\rm vect}(G).$$
Let $c = \chi_{\rm vect}(G),$ then we have that $G \stackrel{+}{\to} K_c.$ By \cite[Theorem~10]{cmrssw}, we have that $\vartheta^+(\overline{G}) \le \vartheta^+(\overline{K_c}) = c.$
\end{proof}

\section{Conclusions and Questions}

If Tsirelson's conjectures about the outcomes of finite dimensional spatial and commuting models for quantum correlations agree, {\it i.e.}, if the non-relativistic and relativistic models give the same outcomes, then 
for every graph one would have $\chi_{\qc}(G)=\chi_{\qmin}(G)= \chi_{\qs}(G) =\chi_{\rm q}(G).$
Thus, determining equality of any pair of these quantities would be interesting.
On the other hand, if the fundamental construction that Tsirelson was attempting to carry out, {\it i.e.}, constructing commuting matrices on a finite dimensional space that realized certain families of vectors, then it would follow that $\chi_{\rm vect}(G) = \chi_{\rm q}(G).$

Thus, proving equality or inequality for any pair of these quantities would be interesting and likely important.

Some other natural questions/problems are:

\begin{question}\label{q_kc}
Does the equality $\chi_{\qc}(G) = \chi_{\qmin}(G)$ for every graph $G$ imply the Kirchberg Conjecture, {\it i.e.}, an affirmative answer to the Connes Embedding Problem?
\end{question}

\begin{question}
Characterise the dual cone in $(\cl S(n,c)\otimes_{\comm}\cl S(n,c))^d$ as a subset of $\cl M(n,c)$.
\end{question}

\end{document}